\documentclass[preprint,12pt]{elsarticle}

\usepackage{amsmath,amssymb,amsthm,mathtools}
\usepackage{booktabs}
\usepackage{graphicx}
\usepackage[section]{placeins}
\usepackage{microtype}
\usepackage{xurl}
\usepackage[hidelinks]{hyperref}
\hypersetup{
  pdftitle={Havel--Hakimi Residues of Common-Divisor Graphs: Complete Asymptotics and Prime-Counting Structure},
  pdfauthor={Randy Davila},
  pdfkeywords={Havel--Hakimi residue, common-divisor graph, prime-counting function, Caro--Wei bound, Riemann hypothesis}
}

\newcommand{\R}{\operatorname{R}}
\newcommand{\CW}{\operatorname{CW}}
\newcommand{\cP}{\mathcal{P}}
\newcommand{\cC}{\mathcal{C}}
\newcommand{\Li}{\operatorname{Li}}

\newtheorem{theorem}{Theorem}
\newtheorem{proposition}[theorem]{Proposition}
\newtheorem{lemma}[theorem]{Lemma}
\newtheorem{corollary}[theorem]{Corollary}
\newtheorem{conjecture}[theorem]{Conjecture}

\journal{European Journal of Combinatorics}

\begin{document}
\frenchspacing
\begin{frontmatter}

\title{Havel--Hakimi Residues of Common-Divisor Graphs:\\
Complete Asymptotics and Prime-Counting Structure}

\author[fp,rice]{Randy Davila\corref{cor1}}
\ead{randy@firstprinciples.com}
\cortext[cor1]{Corresponding author.}
\address[fp]{FirstPrinciples Inc., 77 King Street West, Suite 400,
Toronto, Ontario M5K 0A1, Canada}
\address[rice]{Department of Computational Applied Mathematics \&
Operations Research, Rice University, Houston, Texas, USA}

\begin{abstract}
Let $G_n$ be the graph on $\{2,\ldots,n\}$ in which two integers are adjacent
when they have a common divisor greater than one. We determine the complete
asymptotic expansion of its Havel--Hakimi residue $\R(G_n)$, confirming a
leading-constant prediction of Staton recorded in Fajtlowicz's
\emph{Written on the Wall}. If
$A=\sum_{k=2}^{\infty}(\log k)/(k^2(k-1))$, then the first two terms are
$\R(G_n)=(\zeta(2)-1)n/\log n+(\zeta(2)-1-A)n/\log^2n
+O(n/\log^3n)$. More precisely, the difference between $\R(G_n)$ and the
prime-vertex contribution to the Caro--Wei sum is
$O_\beta(n\exp\{-(\log n)^\beta\})$ for every fixed $0<\beta<1/2$; this
estimate yields every coefficient in the expansion. The upper bound follows
from a degree-preserving realization in which almost all relevant prime
vertices are partitioned into cliques. We also prove that the unlabeled graph
determines $\pi(n)$ through its simplicial true-twin classes. Stable inverses
for weighted sums of the resulting degree-class counts give criteria equivalent to the
Riemann hypothesis, including one involving only the Caro--Wei sum. An exact
local-defect identity additionally reduces the conjectured sharp $+2$ residue
bound to explicit prefix estimates.
\end{abstract}


\end{frontmatter}

\clearpage
\begin{center}
\emph{Dedicated to the memory of Paul Erd\H{o}s and William Staton, whose
ideas initiated the asymptotic residue problem studied here, with special
gratitude to Bill Staton, whose friendship first welcomed the author into the
mathematical community.}
\end{center}

\section{Introduction}

For a finite simple graph $G$, the \emph{Havel--Hakimi residue} $\R(G)$ is
obtained by repeatedly laying off the largest term of the degree sequence.
At each step, the largest term $d$ is deleted, the next $d$ terms are reduced
by one, and the resulting sequence is reordered. The number of zeros that
remain when no positive term remains is the residue. Thus $\R(G)$ depends
only on the degree sequence, not on the particular graph realizing it.
The underlying reduction is the classical Havel--Hakimi procedure
\cite{Havel1955,Hakimi1962}; later work on the residue includes, for example,
the exact analysis for unigraphs in \cite{Barrus2012}.
Fajtlowicz's program
\emph{Graffiti} conjectured that
\(
 \R(G)\leq\alpha(G),
\)
where $\alpha(G)$ is the independence number. Favaron, Mah\'eo, and Sacl\'e
proved this conjecture and established several stronger degree-sequence
consequences; shorter proofs followed
\cite{FavaronMaheoSacle1991,GriggsKleitman1994}. The result is one of the
early examples of a computer-generated conjecture leading to a substantial
theorem in graph theory.

The present paper concerns the arithmetic family
\[
 V(G_n)=\{2,\ldots,n\},\qquad
 \{x,y\}\in E(G_n)\quad\Longleftrightarrow\quad \gcd(x,y)>1.
\]
These graphs have been called non-coprime graphs or common-divisor graphs.
They are closely related to divisor graphs and to the complementary coprime
graphs; see, for example,
\cite{IranmaneshPraeger2010,Knill2016,RaviDesikan2023}.
Write $d_G(v)$ for the degree of $v$, $N_G(v)$ for its open neighborhood,
and $N_G[v]=N_G(v)\cup\{v\}$ for its closed neighborhood. We omit the
subscript when the graph is clear. As usual, $\alpha(G)$ denotes the largest
cardinality of an independent vertex set, $G[X]$ the subgraph induced by
$X\subseteq V(G)$, and $\pi(x)$ the number of primes not exceeding $x$.

The prime vertices form an independent set in $G_n$. Conversely, assign to
each vertex in an independent set one of its prime divisors, assigning a prime
vertex to itself. Two vertices cannot receive the same prime, because they
would then be adjacent. This injection into the primes at most $n$ proves
\(
 \alpha(G_n)=\pi(n),
\)
an observation appearing explicitly in work of Rilwan et al.
\cite{RilwanSameemaOli2020} and in Fajtlowicz's historical account
\cite{FajtlowiczExamples}.
The second degree-sequence statistic used throughout is the \emph{Caro--Wei
sum}
\[
 \CW(G)=\sum_{v\in V(G)}\frac{1}{d_G(v)+1}.
\]

\paragraph{Historical provenance.}
The residue problem for this family has a longer history. Conjecture 448 of
Fajtlowicz's evolving collection \emph{Written on the Wall} asks whether
\(\R(G_n)\geq\sqrt n\) \cite{FajtlowiczWall}. Notes appended to the
conjecture record that Erd\H{o}s and Staton obtained
\[
 \R(G_n)\geq
 \left(\frac{\pi^2}{6}-1-o(1)\right)\frac{n}{\log n}.
\]
They also record Fajtlowicz's earlier proposal of the constant $2/3$,
computations through $n=10{,}000$, and Staton's heuristic that the lower
constant $\pi^2/6-1$ is the correct one. Erd\H{o}s asked for any constant
$c>0$ such that
\(
 \R(G_n)\leq(1-c)n/\log n.
\)
Fajtlowicz later described this as an Erd\H{o}s problem arising from
\emph{Graffiti} \cite{FajtlowiczExamples}. An undated bibliographic entry
records a preprint on this problem by Chung, Erd\H{o}s, Lagarias, and Staton,
but, as of August 2026, no manuscript or theorem statement is linked from
that entry \cite{ChungListing}. We make no claim about its contents. Our
novelty comparisons are therefore restricted to the accessible sources
cited here; among those sources, neither a matching asymptotic upper bound
nor a complete asymptotic expansion appears with a published proof.

No full written proof of the Erd\H{o}s--Staton lower bound has been located in
the accessible record. The lower-bound derivation given below was obtained
independently from the prime degree classes and may therefore reconstruct part
of their unpublished argument. The surviving sketch makes such overlap
plausible, but without their manuscript or notes we cannot determine whether
the proofs coincide. We state this possibility explicitly so that an
independent reconstruction is not mistaken for historical priority.

The historical division of credit is important. The asymptotic question,
the lower-bound argument at scale $n/\log n$, and the conjectured leading
constant belong to the Erd\H{o}s--Staton program recorded by Fajtlowicz. The
new contributions here are the matching degree-preserving upper construction,
the complete expansion, the intrinsic graph reconstructions and
prime-counting criteria, and the residue-defect theory. The historical and
novelty claims made here rely only on the accessible sources cited above.

Our main result resolves the asymptotic problem and determines every term.
Write
\begin{equation}\label{eq:constants}
 c_0=\zeta(2)-1,
 \qquad
 A=\sum_{k=2}^{\infty}\frac{\log k}{k^2(k-1)}.
\end{equation}
Here $\zeta$ denotes the Riemann zeta function, and all logarithms are
natural. Implied constants in $O_\beta(\cdot)$ and $O_M(\cdot)$ may depend
on the displayed subscript but not on $n$.
For $r\geq0$, define
\begin{equation}\label{eq:all-order-coefficients}
 c_r=r!\left(
 1-\sum_{k=2}^{\infty}\frac{1}{k^2(k-1)}
       \sum_{t=0}^{r}\frac{(\log k)^t}{t!}
 \right).
\end{equation}
Thus $c_0=\zeta(2)-1$ and $c_1=\zeta(2)-1-A$.
Let $\cP(n)$ denote the contribution of the prime vertices to
$\CW(G_n)$; its exact definition is given in \eqref{eq:prime-mass}.

\begin{theorem}\label{thm:main}
For every fixed $0<\beta<1/2$,
\begin{equation}\label{eq:stretched-prime-mass}
 0\leq \R(G_n)-\cP(n)
 =O_\beta\!\left(n\exp\{- (\log n)^\beta\}\right).
\end{equation}
Consequently, for every fixed integer $M\geq1$,
\begin{equation}\label{eq:complete-expansion}
 \R(G_n)=n\sum_{r=0}^{M-1}\frac{c_r}{\log^{r+1}n}
 +O_M\!\left(\frac{n}{\log^{M+1}n}\right).
\end{equation}
In particular,
\begin{equation}\label{eq:main-expansion}
 \R(G_n)=
 c_0\frac{n}{\log n}
 +(c_0-A)\frac{n}{\log^2 n}
 +O\!\left(\frac{n}{\log^3 n}\right).
\end{equation}
Consequently,
\begin{equation}\label{eq:ratio-expansion}
 \frac{\R(G_n)}{\pi(n)}
 =c_0-\frac{A}{\log n}
 +O\!\left(\frac{1}{\log^2 n}\right).
\end{equation}
\end{theorem}

The stretched-exponential estimate is stronger than any fixed-order
asymptotic statement and is the quantitative core of the paper. Its proof
separates an arithmetic lower estimate from a graph-theoretic upper
construction. A prime $p$ with
\(\lfloor n/p\rfloor=k\) has degree $k-1$. Its contribution to the
Caro--Wei sum is therefore $1/k$. Summing these contributions gives an exact
prime-interval formula whose complete asymptotic expansion is given by
\eqref{eq:prime-mass-complete}; the total contribution of the composite
vertices is only $O(\sqrt n)$. Since Favaron et al. proved that the residue
dominates the Caro--Wei sum, this yields the lower estimate in Section~2.

The upper estimate requires a different use of the degree sequence. The
residue depends only on that sequence and is a lower bound for the
independence number of every graph realizing it. We construct a realization
in which almost all primes of degree $k-1$ lie in disjoint copies of $K_k$.
The removed prime-to-composite incidences are restored by switches inside
classes of composites having a common least prime factor. Those classes
remain cliques with a matching removed and hence contribute negligibly to an
independent set. This turns Staton's informal clique heuristic into a
degree-preserving argument in Section~3.

The remaining sections develop two consequences of this proof. Section~4
returns to the Havel--Hakimi trajectory itself and expresses
$\R(G_n)-\CW(G_n)$ as a sum of nonnegative local defects. An absorbing
sequence of unit-band steps and an estimate determined by the initial degree
order statistics reduce the observed
sharp $+2$ inequality to an explicit estimate for the reductions preceding
the first unit-band step.
Sections~5 and~6 then ask how much prime-counting information survives when
the integer labels are discarded. The unlabeled graph determines $\pi(n)$,
and stable inverses for its simplicial twin-degree profiles preserve the
square-root error scale equivalent to the Riemann hypothesis. The Caro--Wei
sum yields a further criterion depending only on the degree sequence. Thus
the unconditional asymptotic theorem remains logically separate from the
finer, conditional prime-counting criteria.

The weighted degree-class construction also produces a one-parameter graph
polynomial interpolating between the isolated-prime class and the full prime
count. To
the best of the author's knowledge, neither this simplicial twin-degree
polynomial nor the associated theorems for weighted degree-class sums have appeared
previously.

For clarity, the logical status of the results is as follows. The complete
asymptotic expansion, the degree-preserving realization, the intrinsic
reconstruction of the prime count, and the weighted degree-class criteria are
unconditional theorems. The Riemann-hypothesis statements are equivalences,
not assumptions used in those theorems. Only the bounded residue defect and
its two proposed sufficient estimates remain conjectural.

\begin{figure}[htbp]
\centering
\includegraphics[width=\textwidth]{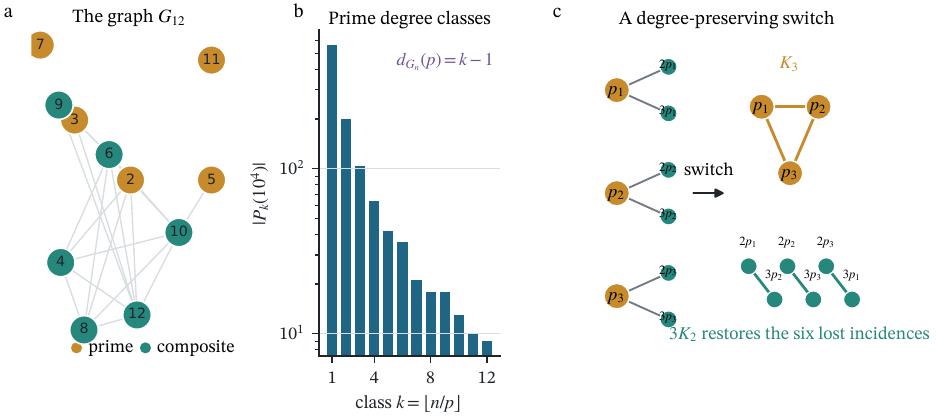}
\caption{Arithmetic structure used in the proof.
\textup{(a)} The common-divisor graph $G_{12}$; prime vertices are highlighted.
\textup{(b)} Prime vertices organize by the integer value of
$\lfloor n/p\rfloor$, which is exactly one more than their degree.
\textup{(c)} A local switch replaces three disjoint two-edge prime stars by a
prime triangle and three composite edges. Every displayed vertex retains its
degree; the proof performs compatible switches inside composite
least-prime-factor classes.}
\label{fig:mechanism}
\end{figure}

\section{Arithmetic structure and the prime-vertex Caro--Wei contribution}

We first isolate the part of the Caro--Wei sum contributed by prime
vertices. Its degree classes admit an exact prime-interval formula, whereas
the entire composite contribution is only $O(\sqrt n)$. This separation
produces the lower estimate in Theorem~\ref{thm:main} and all coefficients
in its asymptotic expansion.

The classical Caro--Wei theorem gives $\CW(G)\leq\alpha(G)$
\cite{Caro1979,Wei1981}. The stronger residue comparison needed here is due
to Favaron et al. \cite{FavaronMaheoSacle1991}:
\begin{equation}\label{eq:fms-chain}
 \CW(G)\leq\R(G)\leq\alpha(G).
\end{equation}

For $k\geq1$, define the $k$th prime degree class by
\[
 P_k(n)=\left\{p\leq n:p\text{ is prime and }
                    \left\lfloor\frac{n}{p}\right\rfloor=k\right\},
 \qquad m_k(n)=|P_k(n)|.
\]
Thus
\begin{equation}\label{eq:mk}
 m_k(n)=\pi(n/k)-\pi(n/(k+1)).
\end{equation}
If $p\in P_k(n)$, then its neighbors are $2p,\ldots,kp$, so
\begin{equation}\label{eq:prime-degree}
 d_{G_n}(p)=k-1.
\end{equation}
Let
\begin{equation}\label{eq:prime-mass}
 \cP(n)=\sum_{p\leq n}\frac{1}{d_{G_n}(p)+1}
       =\sum_{k\geq1}\frac{m_k(n)}{k}
\end{equation}
denote the contribution of the prime vertices to $\CW(G_n)$.

\begin{proposition}\label{prop:prime-mass-exact}
For every $n\geq2$,
\begin{equation}\label{eq:prime-mass-pi}
 \cP(n)=\pi(n)-
 \sum_{k=2}^{\lfloor n/2\rfloor}
 \frac{\pi(n/k)}{k(k-1)}.
\end{equation}
Moreover,
\begin{equation}\label{eq:prime-mass-asymptotic}
 \cP(n)=
 c_0\frac{n}{\log n}
 +(c_0-A)\frac{n}{\log^2 n}
 +O\!\left(\frac{n}{\log^3 n}\right).
\end{equation}
More generally, for every fixed integer $M\geq1$,
\begin{equation}\label{eq:prime-mass-complete}
 \cP(n)=n\sum_{r=0}^{M-1}\frac{c_r}{\log^{r+1}n}
 +O_M\!\left(\frac{n}{\log^{M+1}n}\right).
\end{equation}
\end{proposition}

\begin{proof}
Substituting \eqref{eq:mk} into \eqref{eq:prime-mass} and summing by parts
gives
\[
 \sum_{k\geq1}\frac{\pi(n/k)-\pi(n/(k+1))}{k}
 =\pi(n)-\sum_{k\geq2}\frac{\pi(n/k)}{k(k-1)},
\]
which is \eqref{eq:prime-mass-pi}.

Set $L=\log n$. The Prime Number Theorem, in its standard asymptotic
expansion, gives \cite[Chapter~7]{MontgomeryVaughan2007}
\[
 \pi(x)=\frac{x}{\log x}+\frac{x}{\log^2 x}
        +O\!\left(\frac{x}{\log^3 x}\right).
\]
For $k\leq\sqrt n$, expansion about $L$ yields
\[
 \pi(n/k)=\frac{n}{kL}
 +\frac{n(1+\log k)}{kL^2}
 +O\!\left(\frac{n(1+\log k)^2}{kL^3}\right).
\]
Here $n/k\geq\sqrt n$, so the prime-number-theorem remainder is uniform over
the entire range, and $(\log k)/L\leq1/2$ makes the displayed expansion
uniform as well. The error remains summable after multiplication by
$1/[k(k-1)]$.
For $k>\sqrt n$, the elementary estimate $\pi(n/k)\leq n/k$ shows that the
tail in \eqref{eq:prime-mass-pi} is $O(1)$. Replacing the resulting finite
coefficient sums by their infinite counterparts changes the expansion by
$O(1/\log n)$, which is absorbed by the displayed remainder. Finally,
\[
 \sum_{k=2}^{\infty}\frac{1}{k^2(k-1)}=2-\zeta(2),
 \qquad
 \sum_{k=2}^{\infty}\frac{\log k}{k^2(k-1)}=A.
\]
Substitution in \eqref{eq:prime-mass-pi} proves
\eqref{eq:prime-mass-asymptotic}.

For arbitrary fixed $M\geq1$, use the standard expansion
\[
 \pi(x)=x\sum_{\ell=0}^{M-1}
 \frac{\ell!}{\log^{\ell+1}x}
 +O_M\!\left(\frac{x}{\log^{M+1}x}\right).
\]
Uniformly for $k\leq\sqrt n$, expand each power of
$(\log n-\log k)^{-1}$ through total order $M$. The coefficient of
$n/\log^{r+1}n$ in $\pi(n/k)$ is
\[
 \frac{r!}{k}\sum_{t=0}^{r}\frac{(\log k)^t}{t!}.
\]
After multiplication by $1/[k(k-1)]$, the remainders are summable because
\(
 \sum_{k\geq2}(1+\log k)^M/[k^2(k-1)]<\infty.
\)
The range $k>\sqrt n$ contributes $O(1)$ as above. Substitution in
\eqref{eq:prime-mass-pi} gives the coefficients in
\eqref{eq:all-order-coefficients} and proves
\eqref{eq:prime-mass-complete}.
\end{proof}

The composite contribution is lower order for a simple arithmetic reason.

\begin{lemma}\label{lem:composite-mass}
If $\cC_n$ is the set of composite vertices of $G_n$, then
\[
 \sum_{x\in\cC_n}\frac{1}{d_{G_n}(x)+1}=O(\sqrt n).
\]
Consequently, $\CW(G_n)$ has the same expansion as $\cP(n)$ through the
term of order $n/\log^M n$ for every fixed $M\geq1$; equivalently, the
composite contribution is absorbed by the remainder in every fixed
truncation of the complete expansion.
\end{lemma}

\begin{proof}
Let $q$ be the least prime divisor of a composite integer $x\leq n$. Then
$q\leq\sqrt n$, and the closed neighborhood of $x$ contains every multiple
of $q$ not exceeding $n$. Hence
\[
 d_{G_n}(x)+1\geq\left\lfloor\frac{n}{q}\right\rfloor
 \geq\frac{n}{2q}
\]
for $n\geq4$. Each composite vertex therefore contributes at most
$2/\sqrt n$, and there are fewer than $n$ such vertices.
The cases $n=2,3$ are immediate.
\end{proof}

Equations \eqref{eq:fms-chain} and
\eqref{eq:prime-mass-complete} already recover the Erd\H{o}s--Staton lower
constant and provide every lower-order coefficient. The remainder of the
proof supplies a matching upper estimate with stretched-exponential error.

\section{A degree-preserving realization}

The structural core of the upper estimate is a new realization of the degree
sequence of $G_n$: a graph on the same vertex set having the same degree at
every vertex. We separate the argument into degree-preserving operations on
blocks of prime vertices, the construction of auxiliary composite sets,
simultaneous degree repair, and the final independence bound.

\subsection{Degree-preserving operations on prime blocks}

Fix $0<\beta<1/2$ and set
\[
 K=\left\lfloor\exp\{(\log n)^\beta\}\right\rfloor.
\]
For sufficiently large $n$, every prime in $P_k(n)$ with $k\leq K$ is larger
than $K$, because it exceeds $n/(K+1)$.

\begin{lemma}\label{lem:block-rewiring}
Fix $k\leq K$ and distinct primes
$B=\{p_1,\ldots,p_k\}\subseteq P_k(n)$. Delete the edges
$\{p_i,rp_i\}$ for $1\leq i\leq k$ and $2\leq r\leq k$, and add every edge
inside $B$. Then every prime in $B$ has its original degree, while every
vertex $rp_i$, $2\leq r\leq k$, has degree one less than in $G_n$.

For each $1\leq s\leq\lfloor(k-1)/2\rfloor$ and each $i$, with subscripts
read modulo $k$, add
\begin{equation}\label{eq:direct-switch}
 \bigl\{(2s)p_i,(2s+1)p_{i+1}\bigr\}.
\end{equation}
If $k$ is odd, these additions restore every deficient degree. If $k$ is
even, they restore every deficient degree except those of
$kp_1,\ldots,kp_k$; each of these remaining vertices is deficient by one.
All edges added in \eqref{eq:direct-switch} are new edges of a simple graph.
\end{lemma}

\begin{proof}
Each prime $p_i$ originally has the $k-1$ neighbors
$2p_i,\ldots,kp_i$. It loses these edges and gains the $k-1$ edges joining
it to the other vertices of $B$. This proves the first assertion.

For a proposed edge in \eqref{eq:direct-switch}, any common prime divisor of
its endpoints would divide both consecutive multipliers, or one of the
distinct block primes would divide the opposite multiplier. The first is
impossible, and the second is impossible because $p_i,p_{i+1}>K\geq2s+1$.
Thus the endpoints are coprime, so the edge is not present in $G_n$. As $i$
runs cyclically, every vertex with multiplier $2s$ or $2s+1$ occurs exactly
once. For odd $k$, the multiplier pairs are
$(2,3),(4,5),\ldots,(k-1,k)$; for even $k$, they are
$(2,3),(4,5),\ldots,(k-2,k-1)$, leaving precisely the multiplier-$k$
vertices deficient.
\end{proof}

\subsection{Auxiliary sets and simultaneous degree repair}

Lemma~\ref{lem:block-rewiring} completes an odd prime block locally. An even
block leaves one deficient composite endpoint for each of its $k$ primes.
We repair those endpoints in pairs by deleting edges from a large auxiliary
set of composites. The next definition constructs that auxiliary set
while keeping it disjoint from every endpoint already used.

For the remaining even classes, define $q_k$ to be the least prime that does
not divide $k$. Let
\[
 \mathcal F_n=
 \{rp:p\in P_j(n),\ 1\leq j\leq K,\ 2\leq r\leq j\}
\]
be the set of composite endpoints used by the prime-block operations, and put
\[
 \mathcal S_k(n)=
 \{q_kt\leq n:t\geq2,\ \gcd(t,k)=1\}\setminus\mathcal F_n.
\]

\begin{lemma}\label{lem:reservoir}
Uniformly for even $k\leq K$, every member of $\mathcal S_k(n)$ is composite,
has least prime factor $q_k$, and is coprime to $k$ and to every prime in
$P_k(n)$. Moreover, for all sufficiently large $n$,
\begin{equation}\label{eq:reservoir-large}
 |\mathcal S_k(n)|>2\pi(n)+2.
\end{equation}
\end{lemma}

\begin{proof}
Every prime smaller than $q_k$ divides $k$. If $\gcd(t,k)=1$, none of those
smaller primes divides $t$, while $q_k$ divides $q_kt$. Hence $q_k$ is the
least prime factor of $q_kt$. The condition $t\geq2$ makes the product
composite, and $q_kt$ is coprime to $k$ because both $q_k$ and $t$ are
coprime to $k$.

Suppose that $p\in P_k(n)$ divides a member $c$ of the set before
$\mathcal F_n$ is removed.
Since $c\leq n$ and $\lfloor n/p\rfloor=k$, we have $c=rp$ for some
$1\leq r\leq k$. The vertex $c$ is composite, so $r\geq2$ and
$c\in\mathcal F_n$, contrary to $c\in\mathcal S_k(n)$. Thus the stated
coprimality holds.

Counting reduced residue classes in complete blocks of length $k$ gives
\begin{equation}\label{eq:raw-reservoir}
 \bigl|\{q_kt\leq n:t\geq2,\ \gcd(t,k)=1\}\bigr|
 =\frac{n}{q_k}\frac{\varphi(k)}{k}+O(k).
\end{equation}
Let $\vartheta(y)=\sum_{p\leq y}\log p$ be Chebyshev's first function and
let $\varphi$ be Euler's totient function. Standard estimates for these
functions
give
\[
 q_k\ll\log(2k),
 \qquad
 \frac{\varphi(k)}{k}\gg\frac{1}{\log\log(3k)};
\]
see, for example, \cite[Chapters~1 and~2]{MontgomeryVaughan2007}. For the
first estimate, if every prime below $y$ divides $k$, then
$\log k\geq\vartheta(y)$. Consequently, the quantity in
\eqref{eq:raw-reservoir} is bounded below, uniformly for $k\leq K$, by a
positive constant times
\begin{equation}\label{eq:reservoir-lower-scale}
 \frac{n}{\log(2K)\log\log(3K)}.
\end{equation}

Since $j\leq K=n^{o(1)}$, the Brun--Titchmarsh inequality
\cite[Chapter~3]{MontgomeryVaughan2007}, applied to
$(n/(j+1),n/j]$ gives, uniformly in this range,
$m_j(n)\ll n/[j(j+1)\log n]$. Hence
\begin{equation}\label{eq:forbidden-count}
 |\mathcal F_n|
 \leq\sum_{j\leq K}(j-1)m_j(n)
 \ll\frac{n\log K}{\log n}.
\end{equation}
Now $\log K=(\log n)^\beta+o(1)$. The ratio of the right side of
\eqref{eq:forbidden-count} to \eqref{eq:reservoir-lower-scale} is
\[
 O\!\left(
  \frac{(\log K)^2\log\log(3K)}{\log n}
 \right)=o(1)
\]
because $2\beta<1$. Likewise,
$\pi(n)=O(n/\log n)$ is little-oh of
\eqref{eq:reservoir-lower-scale}. The error $O(k)$ in
\eqref{eq:raw-reservoir} is also negligible uniformly for $k\leq K$.
Removing $\mathcal F_n$ from the set counted in
\eqref{eq:raw-reservoir} proves
\eqref{eq:reservoir-large}.
\end{proof}

\begin{lemma}\label{lem:simultaneous-repair}
After applying Lemma~\ref{lem:block-rewiring} to disjoint blocks in every
class $P_k(n)$, $k\leq K$, all remaining deficient vertices can be repaired
by degree-preserving switches. The auxiliary composite edges deleted in
these switches have pairwise disjoint endpoints.
\end{lemma}

\begin{proof}
Only even classes leave deficient vertices. Within each even block, pair the
vertices $kp_1,\ldots,kp_k$ arbitrarily. Across all classes there are at most
$\pi(n)$ deficient vertices and hence at most $\pi(n)/2$ pairs.

Consider a pair $a=kp_i$ and $b=kp_j$. By
Lemma~\ref{lem:reservoir}, fewer than $\pi(n)$ auxiliary vertices have been
used before any stage. Thus, even after excluding all of them, there are two
distinct unused vertices
$c,d\in\mathcal S_k(n)$. They have the common least prime factor $q_k$, so
$cd$ is an edge. On the other hand,
\[
 \gcd(a,c)=\gcd(b,d)=1:
\]
the auxiliary vertices are coprime to $k$ and to all primes in $P_k(n)$.
Thus $ac$ and $bd$ are absent from the original graph. They also have not
been introduced by an earlier switch, because $a,b,c,d$ have not previously
served as endpoints of an auxiliary repair and $c,d\notin\mathcal F_n$.
Delete $cd$ and add $ac$ and $bd$. The deficient vertices $a,b$ each gain
one incident edge, while $c,d$ each lose one and gain one. All four degrees
are therefore restored. Choosing new auxiliary vertices at each step makes
the deleted edges endpoint-disjoint.
\end{proof}

The construction will be used globally, so we record its compatibility
properties explicitly. First, multiplier endpoints belonging to distinct
processed primes are distinct. Second, every edge inserted by a local
prime-block operation joins vertices that were nonadjacent in $G_n$. Third,
auxiliary
vertices are chosen outside $\mathcal F_n$ and are used in at most one
repair, so no auxiliary deletion or insertion repeats an earlier operation.
Finally, each local block preserves the degrees of its prime vertices, and
each auxiliary repair restores two remaining deficiencies while leaving the
degrees of its auxiliary endpoints unchanged. Thus performing all switches
produces a simple graph and preserves every labeled vertex degree.

\subsection{The global realization and upper bound}

\begin{lemma}\label{lem:realization}
For all sufficiently large $n$, there is a graph $H_n$ on
$\{2,\ldots,n\}$ satisfying $d_{H_n}(v)=d_{G_n}(v)$ for every vertex $v$ and
\begin{equation}\label{eq:realization-bound}
 \alpha(H_n)\leq
 \sum_{k=1}^{K}\left\lfloor\frac{m_k(n)}{k}\right\rfloor
 +\pi\!\left(\frac{n}{K+1}\right)
 +2\pi(\sqrt n)+K^2.
\end{equation}
\end{lemma}

\begin{proof}
For each $k\leq K$, partition all but fewer than $k$ primes in $P_k(n)$ into
blocks of cardinality $k$. The multiplier endpoints belonging to distinct
processed primes are distinct: if $rp=sq$ with distinct processed primes
$p,q>K$ and $1\leq r,s\leq K$, then $p$ divides $s$, which is impossible.
Thus the local prime-block operations are mutually compatible. Apply
Lemma~\ref{lem:block-rewiring} to each block and then apply the simultaneous
repairs of
Lemma~\ref{lem:simultaneous-repair}. The compatibility properties recorded
above show that the resulting graph $H_n$ is simple and has
$d_{H_n}(v)=d_{G_n}(v)$ for every labeled vertex $v$.

Let $\cC_n$ denote the composite vertices. In $G_n$, the composite vertices
with a fixed least prime factor $q\leq\sqrt n$ form a clique. The only
composite-to-composite edges deleted during the construction are the
auxiliary edges $cd$. Their endpoints are globally distinct, so the deleted
edges within each least-prime-factor class form a matching. Thus each such
class induces a clique with at most a matching removed and has independence
number at most two. Added edges cannot increase an independence number;
hence
\begin{equation}\label{eq:composite-alpha}
 \alpha(H_n[\cC_n])\leq2\pi(\sqrt n).
\end{equation}

Every full prime block is a clique and contributes at most one vertex to an
independent set. There are $\lfloor m_k(n)/k\rfloor$ blocks in class $k$.
The unblocked remainders from the first $K$ classes contain fewer than
$\sum_{k\leq K}k\leq K^2$ primes, and every prime in a class with index
greater than $K$ is at most $n/(K+1)$. Adding these bounds to
\eqref{eq:composite-alpha} proves \eqref{eq:realization-bound}.
\end{proof}

\begin{proof}[Proof of Theorem~\ref{thm:main}]
The residue is determined by the degree sequence, so
$\R(H_n)=\R(G_n)$. Applying \eqref{eq:fms-chain} to the realization in
Lemma~\ref{lem:realization} gives
\[
 \R(G_n)\leq\alpha(H_n).
\]
The first term on the right side of \eqref{eq:realization-bound} is at most
$\cP(n)$. The Prime Number Theorem and the definition of $K$ give
\begin{align*}
 \pi\!\left(\frac{n}{K+1}\right)
 &=O_\beta\!\left(n\exp\{-(\log n)^\beta\}\right),\\
 \pi(\sqrt n)+K^2
 &=O_\beta\!\left(n\exp\{-(\log n)^\beta\}\right).
\end{align*}
Therefore
\begin{equation}\label{eq:upper-prime-mass}
 \R(G_n)\leq\cP(n)
 +O_\beta\!\left(n\exp\{-(\log n)^\beta\}\right).
\end{equation}
On the other hand, \eqref{eq:fms-chain} implies
$\R(G_n)\geq\CW(G_n)\geq\cP(n)$. This proves
\eqref{eq:stretched-prime-mass}. Proposition
\ref{prop:prime-mass-exact} then proves \eqref{eq:complete-expansion}, since
$n\exp\{-(\log n)^\beta\}=O_{\beta,M}(n/\log^{M+1}n)$ for every fixed $M$.
Taking $M=2$ gives \eqref{eq:main-expansion}.
Finally,
\(
 \pi(n)=n/\log n+n/\log^2n+O(n/\log^3n),
\)
and division gives \eqref{eq:ratio-expansion}.
\end{proof}

Theorem~\ref{thm:main} gives, in particular, the exact limit proposed by
Staton.

\begin{corollary}\label{cor:historical-limit}
As $n\to\infty$,
\[
 \frac{\R(G_n)\log n}{n}\longrightarrow\zeta(2)-1,
 \qquad
 \frac{\pi(n)}{\R(G_n)}\longrightarrow
 \frac{1}{\zeta(2)-1}=1.550546\ldots.
\]
\end{corollary}

\begin{figure}[htbp]
\centering
\includegraphics[width=\textwidth]{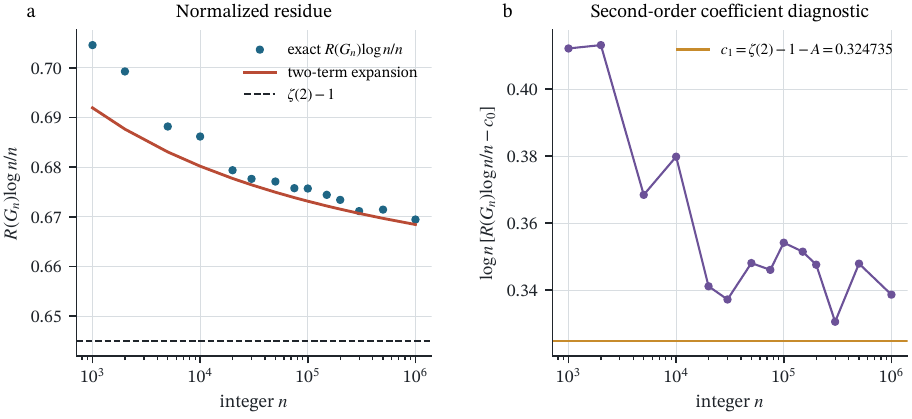}
\caption{Convergence to the zeta constant and its second-order correction.
\textup{(a)} Exact values of $\R(G_n)\log n/n$ approach
$\zeta(2)-1$. The theorem explains both the limiting constant and the visible
drift.
\textup{(b)} Subtracting the leading term and multiplying by $\log n$ extracts
the second coefficient $c_1=\zeta(2)-1-A$. These values are not used in the
proof; they show why the computation directed attention from the prime-counting
function itself to the degree classes indexed by $\lfloor n/p\rfloor$.}
\label{fig:asymptotic}
\end{figure}

\section{A local defect identity for Havel--Hakimi reductions}

The proof above determines the residue to every fixed asymptotic order
without following its individual reduction steps. Exact computations suggest
that the residue is in fact confined to a bounded interval above the
Caro--Wei sum. We now turn to the reduction dynamics themselves, first
deriving a local identity valid for every graphic sequence and then
separating each trajectory into the initial reductions preceding the first
unit-band step and the terminal reductions that follow it.

\subsection{An exact local-defect identity}

A finite sequence of nonnegative integers is \emph{graphic} if it is the
degree sequence of a finite simple graph. Let
$\mathbf d=(d_1,\ldots,d_t)$ be a nonincreasing graphic sequence with
$d_1=D>0$. Its Havel--Hakimi reduction $\mathbf d'$ is obtained by deleting
$D$ and subtracting one from $d_2,\ldots,d_{D+1}$, followed by reordering.
We call $D$ the \emph{pivot} of the step. Graphicality guarantees that these
$D$ selected entries are positive. Define
\(
 \CW(\mathbf d)=\sum_i1/(d_i+1).
\)
Repeated reduction produces a finite trajectory
$\mathbf d^{(0)},\ldots,\mathbf d^{(s)}$ ending in an all-zero sequence; its
number of remaining zeros is denoted $\R(\mathbf d)$. When $\mathbf d$ is
the degree sequence of $G$, these definitions agree with $\CW(G)$ and
$\R(G)$. We call $\R(\mathbf d)-\CW(\mathbf d)$ the
\emph{Caro--Wei defect} of $\mathbf d$; Lemma~\ref{lem:defect-step} expresses
it as a sum of nonnegative increments.

\begin{lemma}\label{lem:defect-step}
With the preceding notation,
\begin{equation}\label{eq:step-defect}
 \CW(\mathbf d')-\CW(\mathbf d)
 =\sum_{i=2}^{D+1}\frac{1}{d_i(d_i+1)}-\frac{1}{D+1}
 \geq0.
\end{equation}
Moreover, the same quantity has the exact heterogeneity representation
\begin{equation}\label{eq:heterogeneity-charge}
 \CW(\mathbf d')-\CW(\mathbf d)
 =\sum_{i=2}^{D+1}
 \frac{(D-d_i)(D+d_i+1)}
 {d_i(d_i+1)D(D+1)}.
\end{equation}
Consequently, if
$\mathbf d^{(0)},\ldots,\mathbf d^{(s)}$ is the complete Havel--Hakimi
trajectory and $\mathbf d^{(s)}$ consists of $\R(\mathbf d)$ zeros, then
\begin{equation}\label{eq:defect-sum}
 \R(\mathbf d)-\CW(\mathbf d)
 =\sum_{j=0}^{s-1}
 \left(
  \sum_{i=2}^{D_j+1}
   \frac{1}{d_i^{(j)}(d_i^{(j)}+1)}
  -\frac{1}{D_j+1}
 \right),
\end{equation}
where $D_j=d_1^{(j)}$.
\end{lemma}

\begin{proof}
Deleting $D$ removes the term $1/(D+1)$. Each of the next $D$ entries is
replaced by $d_i-1$, changing its contribution by
\(
 1/d_i-1/(d_i+1)=1/[d_i(d_i+1)].
\)
This proves the identity. Since $d_i\leq D$ for $2\leq i\leq D+1$, the sum
of the $D$ positive terms is at least $D/[D(D+1)]=1/(D+1)$.
Writing $1/(D+1)$ as the sum of $D$ copies of $1/[D(D+1)]$
and factoring
\[
 D(D+1)-d_i(d_i+1)=(D-d_i)(D+d_i+1)
\]
gives \eqref{eq:heterogeneity-charge}.
Telescoping over the trajectory proves \eqref{eq:defect-sum}.
\end{proof}

Formula \eqref{eq:heterogeneity-charge} shows that the defect is not caused
by the size of a pivot alone. It measures how far the degrees selected by
that pivot fall below the pivot degree. For every graphic sequence, it
therefore separates the defect contributed before the first unit-band step
from the defect contributed by all subsequent steps.

\subsection{Persistence of unit-band steps}

Call a Havel--Hakimi step with pivot $D$ a \emph{unit-band step} if every
selected degree belongs to $\{D,D-1\}$. The condition can be recognized
directly from the current degree multiplicities: if $c_j$ denotes the number
of entries equal to $j$, then the step is a unit-band step precisely when
\begin{equation}\label{eq:unit-band-capacity}
 c_D+c_{D-1}-1\geq D.
\end{equation}

\begin{lemma}\label{lem:unit-band-absorption}
If a Havel--Hakimi trajectory reaches a unit-band step, then every later
step with positive pivot is also a unit-band step.
\end{lemma}

\begin{proof}
At a unit-band step, each of the $D$ selected entries has degree $D$ or
$D-1$ and therefore becomes $D-1$ or $D-2$. Suppose first that the next
maximum is still $D$. An unselected degree-$D$ entry must then remain.
Because the preceding step selected the $D$ largest entries, all of its
selected entries had degree $D$; after reduction they provide $D$ entries of
degree $D-1$. The next selected block consequently uses only degrees $D$ and
$D-1$.

Otherwise, let $D'\leq D-1$ be the next positive pivot. At most one of the
$D$ entries just decremented is removed as that pivot. Hence at least $D-1$
of them remain, and every one has degree at least $D-2\geq D'-1$. Since
$D-1\geq D'$, there are at least $D'$ available entries of degree at least
$D'-1$; none has degree greater than the current maximum $D'$. The next
selected block therefore uses only degrees $D'$ and $D'-1$. Repeating the
argument proves the claim.
\end{proof}

By the \emph{terminal unit-band suffix} we mean the part of the trajectory
beginning with its first unit-band step and ending with the all-zero sequence.
Lemma~\ref{lem:unit-band-absorption} shows that every positive-pivot step in
this suffix is again a unit-band step.

\begin{theorem}\label{thm:terminal-band}
If the maximal terminal unit-band suffix of a graphic Havel--Hakimi
trajectory is nonempty and begins with pivot $D_*$, then its total defect is
at most
\begin{equation}\label{eq:finite-terminal-bound}
 \frac32-\frac1{D_*}-\frac1{D_*+1}<\frac32.
\end{equation}
If the suffix is empty, its total defect is zero. In particular, every
terminal unit-band suffix contributes at most $3/2$.
\end{theorem}

\begin{proof}
Consider a unit-band step with pivot $D\geq2$, and let $s_D$ be the number
of selected entries of degree $D-1$. Substitution into
\eqref{eq:step-defect} gives
\begin{equation}\label{eq:unit-band-charge}
 \delta_D=\frac{2s_D}{D(D^2-1)}.
\end{equation}
For a fixed pivot value $D$, every step except possibly the last one selects
only degree-$D$ entries and therefore contributes zero to the defect. At the
last such step, $0\leq s_D\leq D$; if $s_D>0$, all remaining degree-$D$
entries are selected and no later pivot can equal $D$. Thus the total
positive defect contribution at pivot value $D$ is at most $2/(D^2-1)$. No
pivot in the suffix exceeds
$D_*$, the pivot-$1$ contribution is zero, and the finite telescoping sum
\[
 \sum_{D=2}^{D_*}\frac{2}{D^2-1}
 =\sum_{D=2}^{D_*}\left(\frac{1}{D-1}-\frac{1}{D+1}\right)
 =\frac32-\frac1{D_*}-\frac1{D_*+1}
\]
proves \eqref{eq:finite-terminal-bound}. The empty-suffix assertion is by
definition.
\end{proof}

\subsection{The initial segment and an initial-degree upper bound}

It remains to control the prefix before the first unit-band step. The next
lemma replaces the moving lower edge of each selected block by an order
statistic of the initial degree sequence.

\begin{lemma}\label{lem:static-quantile}
Let
$d_1^{(0)}\geq\cdots\geq d_N^{(0)}$ be a graphic sequence. Immediately
before reduction $j$, let $D_j$ be the pivot and let $\mu_j$ be the least of
the $D_j$ selected degrees. Then
\begin{equation}\label{eq:static-quantile}
 \mu_j\geq d_{D_j+j+1}^{(0)}-j.
\end{equation}
\end{lemma}

\begin{proof}
Among the first $D_j+j+1$ vertices in the initial ordering, at most $j$ have
been removed. After excluding the current pivot, at least $D_j$ remain.
Every surviving degree has been decreased at most once in each preceding
reduction, so each of these vertices currently has degree at least
$d_{D_j+j+1}^{(0)}-j$. Havel--Hakimi selects the $D_j$ largest remaining
degrees, which proves the assertion.
\end{proof}

For $n\geq4$, let $\delta_j$ denote the $j$th summand in
\eqref{eq:defect-sum}. Let $\tau_n$ be the index of the first unit-band step
for $G_n$, with the convention that $\tau_n=s$ if no such step occurs, and
write
\[
 B_n=\sum_{j<\tau_n}\delta_j,
 \qquad
 T_n=\sum_{j\geq\tau_n}\delta_j.
\]
Thus $\R(G_n)-\CW(G_n)=B_n+T_n$, and
Theorem~\ref{thm:terminal-band} gives $0\leq T_n<3/2$ when the suffix is
nonempty and $T_n=0$ otherwise. Put
\begin{equation}\label{eq:static-envelope}
 L_j=\max\{1,d_{D_j+j+1}^{(0)}-j\},
 \qquad
 E_n=\sum_{j<\tau_n}\frac{D_j}{L_j(L_j+1)}.
\end{equation}
The quantity $E_n$ depends on the initial degree order statistics and the
pivots before $\tau_n$; it is the explicit initial-degree upper bound used
below.

\begin{corollary}\label{cor:static-envelope}
For every $n\geq4$,
\begin{equation}\label{eq:two-phase-bound}
 0\leq \R(G_n)-\CW(G_n)\leq E_n+\frac32.
\end{equation}
Consequently, a uniform bound on $E_n$ proves the bounded-defect property
$\R(G_n)-\CW(G_n)=O(1)$.
\end{corollary}

\begin{proof}
All selected degrees are positive. Lemma~\ref{lem:static-quantile} therefore
gives $\mu_j\geq L_j$. Dropping the negative term in
\eqref{eq:step-defect} yields
\[
 \delta_j\leq\frac{D_j}{\mu_j(\mu_j+1)}
 \leq\frac{D_j}{L_j(L_j+1)}.
\]
Summing before $\tau_n$ and applying Theorem~\ref{thm:terminal-band} proves
\eqref{eq:two-phase-bound}.
\end{proof}

\subsection{Conjectures and finite evidence}

The local defects in \eqref{eq:defect-sum} are usually nonzero for $G_n$;
their aggregate nevertheless remains strikingly small. Exact computation
for every $2\leq n\leq10{,}000$ gives
\[
 \Delta_n:=\R(G_n)-\left\lceil\CW(G_n)\right\rceil\leq2.
\]
Selected larger values of $n$ are shown in Table~\ref{tab:data}. These calculations
use exact integer degrees and the deterministic Havel--Hakimi algorithm; only
the displayed decimal evaluation of $\CW(G_n)$ uses floating-point
arithmetic. The Supplementary Material contains the complete trajectory
archive for $4\leq n\leq10{,}000$, selected trajectories through $10^6$, an
independent verifier, and source code that regenerates any selected value of
$n$.

\begin{table}[ht]
\centering
\caption{Residue and Caro--Wei values at selected values of $n$.}
\label{tab:data}
\begin{tabular}{rrrrr}
\toprule
$n$ & $\pi(n)$ & $\R(G_n)$ & $\CW(G_n)$ & $\R(G_n)-\CW(G_n)$\\
\midrule
$1{,}000$     & $168$    & $102$    & $100.456651$ & $1.543349$\\
$10{,}000$    & $1{,}229$& $745$    & $743.380821$ & $1.619179$\\
$100{,}000$   & $9{,}592$& $5{,}869$& $5{,}867.789286$ & $1.210714$\\
$300{,}000$   & $25{,}997$& $15{,}965$& $15{,}963.723228$ & $1.276772$\\
$500{,}000$   & $41{,}538$& $25{,}584$& $25{,}582.255731$ & $1.744269$\\
$1{,}000{,}000$& $78{,}498$& $48{,}456$& $48{,}454.121402$ & $1.878598$\\
\bottomrule
\end{tabular}
\end{table}

The two-phase quantities reveal considerably more structure than the total
defect alone. At each value of $n$ in Table~\ref{tab:two-phase}, a unit-band
step occurs; write $D_*(n)$ for its pivot. The initial contribution $B_n$
stays close to one, while the terminal contribution $T_n$ is bounded
uniformly by Theorem~\ref{thm:terminal-band}. The upper bound $E_n$ is much
larger than the observed value of $B_n$, but it is expressed through the
initial degree order statistics in \eqref{eq:static-envelope}.

\begin{table}[ht]
\centering
\caption{Initial and terminal defect contributions and their
initial-degree upper bound.}
\label{tab:two-phase}
\begin{tabular}{rrrrr}
\toprule
$n$ & $D_*(n)/n$ & $B_n$ & $T_n$ & $E_n$\\
\midrule
$1{,}000$   & $0.146000$ & $0.981341$ & $0.562008$ & $6.847457$\\
$4{,}000$   & $0.145000$ & $1.011173$ & $0.779723$ & $6.930103$\\
$16{,}000$  & $0.142000$ & $1.013505$ & $1.312560$ & $7.014821$\\
$100{,}000$ & $0.143010$ & $1.006434$ & $0.204279$ & $7.015746$\\
$500{,}000$ & $0.143434$ & $1.006945$ & $0.737324$ & $7.015225$\\
$1{,}000{,}000$ & $0.143408$ & $1.007693$ & $0.870905$ & $7.018865$\\
\bottomrule
\end{tabular}
\end{table}

For $n\geq15$ and $k\geq1$, define the contribution from the $k$th reciprocal
pivot interval by
\[
 C_k(n)=
 \sum_{\substack{j<\tau_n\\ n/(k+1)\leq D_j<n/k}}\delta_j,
 \qquad
 M_k=\max_{15\leq n\leq10{,}000}C_k(n).
\]
Thus $B_n=\sum_{k\geq1}C_k(n)$. Only the first seven terms are nonzero in the
displayed finite range.

\begin{figure}[t]
\centering
\includegraphics[width=\textwidth]{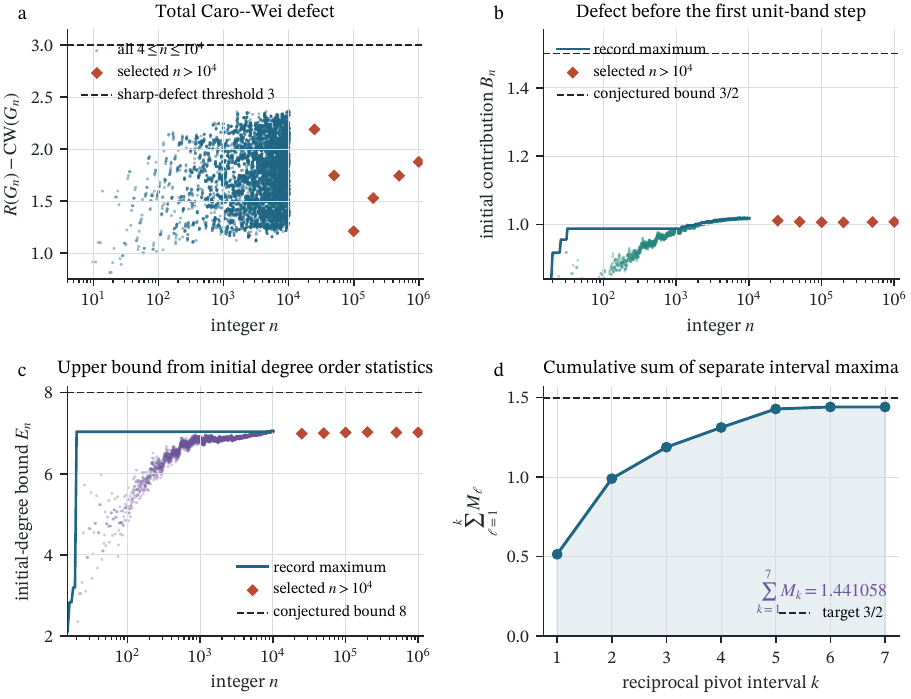}
\caption{Computational evidence for the open defect estimates. Small points
represent every integer through $10^4$; diamonds represent the selected
values $25{,}000$, $50{,}000$, $100{,}000$, $200{,}000$, $500{,}000$, and
$10^6$.
\textup{(a)} The total defect remains below the threshold $3$ that implies
the sharp integer bound in Conjecture~\ref{conj:bounded-defect}.
\textup{(b)} The contribution $B_n$ before the first unit-band step, together
with its record maximum and the proposed bound $3/2$.
\textup{(c)} The initial-degree upper bound $E_n$, its record maximum, and the
proposed bound $8$; panels \textup{(b)} and \textup{(c)} display every
$15\leq n\leq10^4$.
\textup{(d)} The cumulative sums of the separate finite-range maxima $M_k$.
Their sum is $1.441057\ldots<3/2$, although the maxima need not occur at the
same value of $n$.}
\label{fig:defect}
\end{figure}

The strongest pattern in these computations is the following uniform bound
on the difference between the residue and the Caro--Wei lower bound.

\begin{conjecture}\label{conj:bounded-defect}
If $n\geq2$, then
\[
 \left\lceil\CW(G_n)\right\rceil
 \leq\R(G_n)
 \leq\left\lceil\CW(G_n)\right\rceil+2.
\]
\end{conjecture}

\begin{conjecture}\label{conj:prefix-charge}
For every $n\geq15$, the defect contribution before the first unit-band step
satisfies
\[
 B_n\leq\frac32.
\]
\end{conjecture}

\begin{proposition}\label{prop:sharp-prefix-reduction}
Conjecture~\ref{conj:prefix-charge} implies
Conjecture~\ref{conj:bounded-defect}.
\end{proposition}

\begin{proof}
The cases $2\leq n\leq14$ are checked directly. Let $n\geq15$. If the
terminal unit-band suffix is empty, then
$\R(G_n)-\CW(G_n)=B_n\leq3/2$, which already gives the desired integer
upper bound. Otherwise,
Theorem~\ref{thm:terminal-band} and Conjecture~\ref{conj:prefix-charge} give
\[
 \R(G_n)-\CW(G_n)=B_n+T_n<3.
\]
The integer
$\R(G_n)-\lceil\CW(G_n)\rceil$ is therefore strictly smaller than $3$ and
hence is at most $2$. The lower inequality in
Conjecture~\ref{conj:bounded-defect} follows from \eqref{eq:fms-chain}.
\end{proof}

\begin{conjecture}\label{conj:static-envelope}
For every $n\geq4$, the quantity in \eqref{eq:static-envelope} satisfies
\[
 E_n<8.
\]
\end{conjecture}

The lower inequality in Conjecture~\ref{conj:bounded-defect} follows from the
theorem of Favaron et al. and is included to display both proposed bounds.
Its upper inequality is open and
strictly stronger than Theorem~\ref{thm:main}: it asserts that the total
defect in \eqref{eq:defect-sum} is bounded independently of $n$, whereas the
stretched-exponential error in Theorem~\ref{thm:main} still diverges with
$n$. The conjecture also gives a concrete target for a direct
analysis of the Havel--Hakimi dynamics. Conjecture
\ref{conj:static-envelope} is weaker numerically but more directly tied to
the known arithmetic degree distribution: Corollary
\ref{cor:static-envelope} shows that it already proves bounded defect. Exact
calculation gives $\max_{4\leq n\leq2000}E_n=E_{20}=7.03333\ldots$, and all
values through $10{,}000$ satisfy $E_n<7.053$; the selected values at
$16{,}000$, $100{,}000$, and $500{,}000$ are below $7.016$. A proof may
therefore focus on initial degree order statistics rather than on the complete
reordered trajectory.

The quantities $C_k(n)$ partition $B_n$ according to the reciprocal pivot
intervals
\begin{equation}\label{eq:reciprocal-pivot-bands}
 \frac{n}{k+1}\leq D_j<\frac{n}{k}
 \qquad(k=1,2,\ldots).
\end{equation}
For every $15\leq n\leq10{,}000$, the sum of the separate observed maxima is
$1.441058\ldots<3/2$; only the first seven intervals occur, and the seventh
contributes less than $3.9\times10^{-6}$. This does not
prove Conjecture~\ref{conj:prefix-charge}, because maxima over a finite data
set are not uniform estimates. It does, however, reduce the desired proof to
establishing summable arithmetic upper bounds for the band contributions,
rather than estimating every Havel--Hakimi step independently. In the
same exhaustive range, $n=14$ is the unique value with $B_n>3/2$, where
$B_{14}=403/252$ and $T_{14}=0$.

\section{Recovering primes from the unlabeled graph}

We next record an intrinsic reconstruction that explains why
the arithmetic content of $G_n$ is not confined to its vertex labels. A
vertex is \emph{simplicial} if its open neighborhood induces a clique. Two
vertices are \emph{true twins} if their closed neighborhoods are equal.

\begin{proposition}\label{prop:reconstruction}
A vertex of $G_n$ is simplicial if and only if it is a prime power. The
true-twin classes of simplicial vertices are
\[
 \{p,p^2,\ldots,p^a\},
 \qquad a=\max\{j:p^j\leq n\},
\]
one class for each prime $p\leq n$. Consequently, if
\(\operatorname{Tw}(v)=\{u:N[u]=N[v]\}\) denotes the true-twin class of
$v$, then
\[
 \pi(n)=
 \sum_{\substack{v\in V(G_n)\\v\text{ simplicial}}}
 \frac{1}{|\operatorname{Tw}(v)|}.
\]
\end{proposition}

\begin{proof}
If $v=p^a$, every neighbor of $v$ is divisible by $p$, so its neighborhood
is a clique. If two distinct primes $p$ and $q$ divide $v$, then $p$ and $q$
are nonadjacent neighbors of $v$; hence $v$ is not simplicial. The closed
neighborhood of every power of $p$ is precisely the set of multiples of $p$
in $\{2,\ldots,n\}$. Powers of the same prime are therefore true twins,
while powers of different primes are not. Each true-twin class contributes
one to the displayed sum.
\end{proof}

This reconstruction is not proposed as a faster prime-counting algorithm:
constructing the graph is more expensive than a sieve. Its significance is
structural. The isomorphism type of the unlabeled graph determines
$n$ from its order and then determines $\pi(n)$ from the displayed sum,
while its degree sequence alone determines a residue having the
same first-order scale but the nontrivial constant $\zeta(2)-1$.

\section{Riemann-hypothesis-scale information in the graph}

The exact identities above preserve more than the complete asymptotic series.
They also preserve the square-root error scale in the prime number theorem.
For real $x\geq2$, write $G_x=G_{\lfloor x\rfloor}$. To make the error scale
precise, let
\[
 \Li(x)=\int_2^x\frac{dt}{\log t}\qquad(x\geq2),
\]
and set $\Li(x)=\pi(x)=0$ for $x<2$. Put $E(x)=\pi(x)-\Li(x)$. The elementary
bound $E(x)=O(x)$ ensures that this function has polynomial growth. The
classical von Koch criterion states that the Riemann hypothesis is equivalent to
\begin{equation}\label{eq:von-koch}
 E(x)=O(\sqrt x\log(ex));
\end{equation}
see, for example, \cite{Titchmarsh1986,MontgomeryVaughan2007}.
Here and below, the Riemann hypothesis means that every nontrivial zero of
$\zeta(s)$ has real part $1/2$.

\subsection{A boundedly invertible dilation operator}

For $0<r<1$, define the dilation operator $D_r f(x)=f(rx)$, with every
function extended by zero below $2$. The following elementary observation
provides the bounded-inverse estimate used throughout this section.

\begin{lemma}\label{lem:dilation-filter}
Let
\[
 B=\sum_{j\in J}b_jD_{r_j},\qquad 0<r_j<1,
\]
where $J$ is finite or countable, the $b_j$ are real or complex scalars, and
\begin{equation}\label{eq:dilation-contraction}
 \kappa=\sum_{j\in J}|b_j|\sqrt{r_j}<1.
\end{equation}
If $f$ has at most polynomial growth, meaning that
$f(x)=O(x^A\log^m(ex))$ for some $A,m\geq0$, then
\[
 f(x)=O(\sqrt x\log(ex))
 \quad\Longleftrightarrow\quad
 (I-B)f(x)=O(\sqrt x\log(ex)).
\]
\end{lemma}

\begin{proof}
Let $W$ be the Banach space of functions $f$ on $[2,\infty)$ for which the weighted
norm
\[
 \|f\|_W=\sup_{x\geq2}\frac{|f(x)|}{\sqrt x\log(ex)}
\]
is finite. The series defining $B$ converges absolutely in operator norm.
If $rx\geq2$, then
\[
 \frac{\sqrt{rx}\log(erx)}{\sqrt x\log(ex)}\leq\sqrt r;
\]
if $rx<2$, the extended function vanishes. Thus
$\|D_r\|_W\leq\sqrt r$, and \eqref{eq:dilation-contraction} gives
$\|B\|_W\leq\kappa<1$. The Neumann series
$(I-B)^{-1}=\sum_{\ell\geq0}B^\ell$ converges in operator norm. This proves
the forward implication. For the converse, put $g=(I-B)f\in W$ and
$h=(I-B)^{-1}g\in W$. Then $u=f-h$ satisfies $u=Bu$. Both $f$ and $h$ have
at most polynomial growth, so there is an exponent $A>1/2$ for which
$\|u\|_A=\sup_{x\geq2}|u(x)|/x^A<\infty$. On this polynomial-growth space,
$\|D_{r_j}\|_A\leq r_j^A\leq\sqrt{r_j}$ and hence
$\|B\|_A\leq\kappa<1$. The identity $u=Bu$ now forces $u=0$, so $f=h\in W$.
\end{proof}

\subsection{Fixed degree classes}

For a fixed integer $k\geq1$, let $s_k(G_x)$ be the number of true-twin
classes of simplicial vertices whose common degree is $k-1$. This statistic
is determined by the unlabeled graph. Proposition~\ref{prop:reconstruction}
identifies it with the number of primes $p\leq x$ for which
$\lfloor x/p\rfloor=k$.
We refer to the sequence $(s_k(G_x))_{k\geq1}$ as the
\emph{simplicial degree profile} of $G_x$.

\begin{theorem}\label{thm:rh-profile}
For every fixed integer $k\geq1$, the Riemann hypothesis is equivalent to
\begin{equation}\label{eq:rh-profile}
 s_k(G_x)=\Li(x/k)-\Li(x/(k+1))+O(\sqrt x\log(ex)).
\end{equation}
In particular, for $k=1$ it is equivalent to
\[
 i(G_x)=\Li(x)-\Li(x/2)+O(\sqrt x\log(ex)),
\]
where $i(G_x)$ is the number of isolated vertices.
\end{theorem}

\begin{proof}
Proposition~\ref{prop:reconstruction} and \eqref{eq:prime-degree} give the
exact identity
\[
 s_k(G_x)=\pi(x/k)-\pi(x/(k+1)).
\]
Put $r=k/(k+1)$. After subtracting the two logarithmic-integral terms,
the error in \eqref{eq:rh-profile} is
\[
 E(x/k)-E(x/(k+1))=(I-D_r)E(x/k).
\]
Since $\sqrt r<1$, Lemma~\ref{lem:dilation-filter} and
\eqref{eq:von-koch} prove the forward implication. Conversely, the estimate
in \eqref{eq:rh-profile}, after replacing $x$ by $kx$ and using that $k$ is
fixed, states that $(I-D_r)E(x)=O(\sqrt x\log(ex))$;
Lemma~\ref{lem:dilation-filter} then gives
\eqref{eq:von-koch}. When $k=1$, the isolated
vertices are exactly the primes in $(x/2,x]$.
\end{proof}

\subsection{Weighted profile transforms}

The fixed-degree-class criteria belong to a larger class of intrinsic graph
statistics. Let $\mathcal C(G)$ be the set of true-twin classes consisting
of simplicial vertices. Simpliciality and degree are constant on each such
class. We define the \emph{simplicial twin-degree polynomial} by
\begin{equation}\label{eq:simplicial-polynomial}
 \mathcal S_G(z)=\sum_{C\in\mathcal C(G)}z^{d(C)},
\end{equation}
where $d(C)$ is the common degree of the vertices in $C$. This is an
isomorphism invariant of every finite graph. Proposition~\ref{prop:reconstruction}
gives the particularly transparent identity
\begin{equation}\label{eq:simplicial-polynomial-primes}
 \mathcal S_{G_x}(z)
 =\sum_{p\leq x}z^{\lfloor x/p\rfloor-1}
 =\sum_{k\geq1}z^{k-1}s_k(G_x).
\end{equation}

We next show that this polynomial is one member of an invertible family of
weighted sums of the simplicial degree profile. For $\theta>0$ and $m\geq0$,
let $W_{\theta,m}$
be the space of functions $f$ on $[2,\infty)$ with
\[
 \|f\|_{\theta,m}
 =\sup_{x\geq2}\frac{|f(x)|}{x^\theta\log^m(ex)}<\infty,
\]
again extending functions by zero below two. Then
$\|D_{1/j}\|_{W_{\theta,m}}\leq j^{-\theta}$.
For a sequence $c=(c_j)_{j\geq1}$, write
\[
 \|c\|_{\ell^1_\theta}=\sum_{j\geq1}\frac{|c_j|}{j^\theta}.
\]
This weighted sequence space is a Banach algebra under Dirichlet convolution
$(c\ast d)_n=\sum_{r\mid n}c_r d_{n/r}$, since
$\|c\ast d\|_{\ell^1_\theta}\leq
\|c\|_{\ell^1_\theta}\|d\|_{\ell^1_\theta}$.

\begin{theorem}\label{thm:weighted-profile}
Let $w=(w_k)_{k\geq1}$ be a real sequence, put $a_1=w_1$, and put
$a_j=w_j-w_{j-1}$ for $j\geq2$. Suppose $a_1\neq0$, $a\in
\ell^1_\theta$, and the Dirichlet inverse $b=a^{(-1)}$, characterized by
$a\ast b=e$ with $e_1=1$ and $e_j=0$ for $j\geq2$, also
belongs to $\ell^1_\theta$.
Define
\[
 \mathcal T_w(G_x)=\sum_{k\geq1}w_k s_k(G_x),
 \qquad
 \mathcal M_w(x)=\sum_{k\geq1}w_k
 \left(\Li(x/k)-\Li(x/(k+1))\right).
\]
For every $m\geq0$,
\begin{equation}\label{eq:weighted-profile-equivalence}
 \pi(x)-\Li(x)=O\!\left(x^\theta\log^m(ex)\right)
\end{equation}
if and only if
\begin{equation}\label{eq:weighted-profile-graph-equivalence}
\mathcal T_w(G_x)-\mathcal M_w(x)
 =O\!\left(x^\theta\log^m(ex)\right).
\end{equation}
The hypotheses on $b$ hold, in particular, whenever
\begin{equation}\label{eq:weighted-profile-contraction}
 \kappa_\theta
 =\frac{1}{|a_1|}\sum_{j=2}^{\infty}\frac{|a_j|}{j^\theta}<1.
\end{equation}
\end{theorem}

\begin{proof}
Both profile sums are finite at each $x$. Summation by parts and the exact
formula for $s_k(G_x)$ give
\[
 \mathcal T_w(G_x)
 =\sum_{j\geq1}a_j\pi(x/j),
 \qquad
 \mathcal M_w(x)=\sum_{j\geq1}a_j\Li(x/j).
\]
Consequently, with $E=\pi-\Li$,
\[
 \mathcal T_w(G_x)-\mathcal M_w(x)
 =\mathcal D_aE(x),
 \qquad
 \mathcal D_c=\sum_{j\geq1}c_jD_{1/j}.
\]
The dilation estimate gives
$\|\mathcal D_c\|_{W_{\theta,m}}\leq\|c\|_{\ell^1_\theta}$, while
$D_{1/j}D_{1/k}=D_{1/(jk)}$. If $E\in W_{\theta,m}$, this proves the forward
implication. Conversely, suppose $\mathcal D_aE\in W_{\theta,m}$. For each
fixed $x$, all sums involving $E(x/j)$ are finite because functions vanish
below $2$. Dirichlet convolution therefore gives the pointwise identity
$\mathcal D_b\mathcal D_aE=\mathcal D_{b\ast a}E=E$, without first assuming
$E\in W_{\theta,m}$. Since $\mathcal D_b$ is bounded on
$W_{\theta,m}$, the converse follows.

For the last assertion, write $a=a_1(e+h)$, where $e$ is the convolution
identity defined above and $h_j=a_j/a_1$ for $j\geq2$. Condition
\eqref{eq:weighted-profile-contraction} says
$\|h\|_{\ell^1_\theta}<1$. Hence
\[
 a^{(-1)}=a_1^{-1}\sum_{r\geq0}(-h)^{\ast r}
\]
converges in $\ell^1_\theta$.
\end{proof}

\begin{corollary}\label{cor:polynomial-profile}
Fix $z\in[0,1]$, $\theta>0$, and $m\geq0$, and define
\[
 \mathcal M_z(x)=\sum_{k\geq1}z^{k-1}
 \left(\Li(x/k)-\Li(x/(k+1))\right).
\]
Then
\[
 \pi(x)-\Li(x)=O\!\left(x^\theta\log^m(ex)\right)
\]
if and only if
\[
 \mathcal S_{G_x}(z)-\mathcal M_z(x)
 =O\!\left(x^\theta\log^m(ex)\right).
\]
In particular, for every fixed $z\in[0,1]$, the latter estimate with
$\theta=1/2$ and $m=1$ is equivalent to the Riemann hypothesis.
\end{corollary}

\begin{proof}
For $w_k=z^{k-1}$ and $z<1$, we have $a_1=1$ and
$a_j=-(1-z)z^{j-2}$ for $j\geq2$. Therefore
\[
 \kappa_\theta
 =(1-z)\sum_{j\geq2}\frac{z^{j-2}}{j^\theta}
 \leq 2^{-\theta}(1-z)\sum_{j\geq2}z^{j-2}
 =2^{-\theta}<1.
\]
Theorem~\ref{thm:weighted-profile} applies. When $z=1$, both statements
coincide because $\mathcal S_{G_x}(1)=\pi(x)$ and
$\mathcal M_1(x)=\Li(x)$.
\end{proof}

Thus the polynomial interpolates between the isolated-vertex statistic
$\mathcal S_{G_x}(0)=s_1(G_x)$ and the exact class count
$\mathcal S_{G_x}(1)=\pi(x)$. For $0\leq z<1$, summation by parts also gives
the exact recursive identity
\begin{equation}\label{eq:profile-recursion}
 \pi(x)=\mathcal S_{G_x}(z)
 +(1-z)\sum_{j\geq2}z^{j-2}\pi(x/j).
\end{equation}
The arguments $x/j$ are smaller than $x$, so a fixed polynomial evaluation
across the nested graph family determines the prime count recursively. This
identity is structural and is not proposed as a faster alternative to a
sieve. Corollary~\ref{cor:polynomial-profile} says more: the transform and
its inverse preserve every positive-power error scale, not only the main term.

\subsection{A degree-sequence criterion}

The weight $w_k=1/k$ gives
\[
 \mathcal T_w(G_x)=\sum_{k\geq1}\frac{s_k(G_x)}{k},
\]
which is exactly the contribution of the prime vertices to the Caro--Wei
sum. Here $a_1=1$ and $a_j=-1/(j(j-1))$ for $j\geq2$, so
\eqref{eq:weighted-profile-contraction} holds for every $\theta>0$. The full
Caro--Wei sum adds the composite-vertex contribution, which is
$O(\sqrt x)$ by Lemma~\ref{lem:composite-mass}. It therefore gives
a distinct statistic determined only by the degree sequence and governed by
the same bounded-inverse argument. Define
\begin{equation}\label{eq:filtered-li}
 \mathcal L_{\rm CW}(x)=\Li(x)-
 \sum_{j=2}^{\lfloor x/2\rfloor}
 \frac{\Li(x/j)}{j(j-1)}.
\end{equation}

\begin{theorem}\label{thm:rh-cw}
The Riemann hypothesis is equivalent to
\begin{equation}\label{eq:rh-cw}
 \CW(G_x)=\mathcal L_{\rm CW}(x)+O(\sqrt x\log(ex)).
\end{equation}
\end{theorem}

\begin{proof}
Let
\[
 Q=\sum_{j=2}^{\infty}\frac{1}{j(j-1)}D_{1/j}.
\]
Proposition~\ref{prop:prime-mass-exact} and
Lemma~\ref{lem:composite-mass} give
\[
 \CW(G_x)=(I-Q)\pi(x)+O(\sqrt x),
 \qquad
 \mathcal L_{\rm CW}(x)=(I-Q)\Li(x).
\]
Consequently,
\[
 \CW(G_x)-\mathcal L_{\rm CW}(x)=(I-Q)E(x)+O(\sqrt x).
\]
Moreover,
\[
 \sum_{j=2}^{\infty}\frac{1}{j(j-1)\sqrt j}
 =0.5716341209\ldots<1.
\]
Lemma~\ref{lem:dilation-filter} therefore shows that
\eqref{eq:rh-cw} is equivalent to \eqref{eq:von-koch}. In the converse
direction, the $O(\sqrt x)$ composite term is subtracted first and is smaller
than the permitted $O(\sqrt x\log(ex))$ error.
\end{proof}

Unlike $\alpha(G_x)=\pi(x)$, Theorem~\ref{thm:rh-cw} requires neither the
integer labels nor a maximum independent set: once the degree sequence is
given, $\CW(G_x)$ is obtained by a single sum. Thus a strongly compressed,
polynomial-time graph statistic retains the full Riemann-hypothesis error
scale.

The bounded-defect problem now has an additional consequence.

\begin{corollary}\label{cor:conditional-residue-rh}
If Conjecture~\ref{conj:bounded-defect} holds, then the Riemann hypothesis is
equivalent to
\[
 \R(G_x)=\mathcal L_{\rm CW}(x)+O(\sqrt x\log(ex)).
\]
More generally, the same conclusion follows from
\[
 \R(G_x)-\CW(G_x)=O(\sqrt x\log(ex)).
\]
\end{corollary}

\begin{proof}
Under either hypothesis, replacing $\CW(G_x)$ by $\R(G_x)$ does not change
the error scale in Theorem~\ref{thm:rh-cw}. Conjecture
\ref{conj:bounded-defect} gives the stronger estimate
$\R(G_x)-\CW(G_x)=O(1)$.
\end{proof}

\section{Computational evidence and AI-assisted methodology}

The historical problem itself came from \emph{Graffiti}, and the present
study deliberately retained that conjecture--example--proof workflow through
the \emph{TxGraffiti}--TheoX line
\cite{Fajtlowicz1988,DavilaTxGraffiti2026}.
TheoX is a software system developed as a continuation of the
\emph{TxGraffiti} automated-conjecturing framework. Under human supervision,
it iterates among exact finite-data generation, automated production of
candidate statements, systematic counterexample searches, and analysis of
possible proofs. For this study, the principal data structure was a finite
machine-readable table. For each integer $n$ in its range, the table recorded
$G_n$, its degree sequence, residue, Caro--Wei sum, prime-class sizes
$m_k(n)$, and the counterexamples that invalidated rejected statements.
These finite data guided the choice of statements; they are not premises in
any proof.

This table first exposed the small observed defect
$\R(G_n)-\CW(G_n)$ and motivated Conjecture~\ref{conj:bounded-defect}.
Searching over increasing values of the prime-class truncation parameter
then suggested
$K=\exp\{(\log n)^\beta\}$; the lower bound for
$|\mathcal S_k(n)|$ imposes
$2\beta<1$, exactly the restriction used in Theorem~\ref{thm:main}.
Comparing the fixed-degree-class and Caro--Wei formulas subsequently led to the
weighted profile in Theorem~\ref{thm:weighted-profile} and the simplicial
twin-degree polynomial. These identities were stated as theorems only after
the associated dilation operators were proved to have bounded inverses on
the required weighted function spaces.

For the defect analysis, a second finite table recorded, at every
Havel--Hakimi step, the change in the Caro--Wei sum, the pivot, the least
selected degree, and the multiplicities of the largest degrees. Factoring the
local increment gave \eqref{eq:heterogeneity-charge}. The computed
trajectories then suggested the unit-band condition that is proved persistent
in Lemma~\ref{lem:unit-band-absorption}. Proportional estimates that failed
near primorial values of $n$ redirected the analysis to the initial
order-statistic bound in Lemma~\ref{lem:static-quantile} and the partition
into reciprocal pivot intervals in \eqref{eq:reciprocal-pivot-bands}. Thus
finite computation isolated the two estimates stated in
Conjectures~\ref{conj:prefix-charge} and \ref{conj:static-envelope}; the
theorems in the paper were proved independently of those computations.

All numerical claims used in figures and tables were checked independently.
Exact rational calculations verify the profile identities and recursive
inverses for every $2\leq n\leq1000$ at
$z\in\{0,1/4,1/2,3/4,9/10,1\}$. Tests on 494 stored graphic sequences and
14,256 randomly generated graphic sequences of order at most 60 found no
violation of the finite terminal formula or the persistence of the unit-band
property. These computations
found indexing errors during development, but no theorem rests on finite
testing.

The workflow was developed at FirstPrinciples Inc. and incorporated OpenAI
language-model services. AI output was treated as a source of candidates and
criticism, not as mathematical authority. The author independently
reconstructed and checked every proof, citation, computation, and passage of
the final manuscript and takes full responsibility for its contents. The
degree-preserving construction in
Lemma~\ref{lem:realization} was independently implemented and checked on
finite instances: every vertex degree is compared before and after the
switches, every proposed new edge is verified to be absent, and every prime
block is verified to induce the required clique. The Supplementary Material
contains this checker, a claim-to-artifact index, the exact profile verifier,
the Havel--Hakimi trajectory generator, dense trajectory tables for every
$4\leq n\leq10{,}000$, selected trajectories through $10^6$, and a
machine-readable verification report.
Additional development records are available from the author on reasonable
request.

\section{Conclusion}

The common-divisor graph problem has developed through several generations
of computer-assisted conjecturing. \emph{Graffiti} first posed a weak
residue bound; Erd\H{o}s and Staton initiated the asymptotic program, found
the correct lower scale, and identified the candidate leading constant;
\emph{TxGraffiti} introduced finite tabular data for automated
conjecturing; and TheoX helped isolate the degree-sequence argument used
here. Theorem~\ref{thm:main} settles the asymptotic constant, gives every
lower-order coefficient explicitly, and proves that the difference between
the residue and the prime-vertex contribution $\cP(n)$ to the Caro--Wei sum is
$O_\beta(n\exp\{-(\log n)^\beta\})$. Its first correction
is governed by the convergent logarithmic series $A$ in
\eqref{eq:constants}. The degree-preserving realization may apply more
broadly to arithmetic graphs with many low-degree vertices organized by
short multiplicative intervals.

Theorems
\ref{thm:rh-profile}--\ref{thm:rh-cw} further show that an algebra of
weighted sums of the simplicial degree profile retain the complete square-root error
scale equivalent to the Riemann hypothesis. In particular, the simplicial
twin-degree polynomial supplies a continuum of such criteria, while the
Caro--Wei sum gives a distinct statistic determined solely by the degree
sequence. These criteria
show that substantial prime-counting information survives after the integer
labels and, in the Caro--Wei case, all graph structure beyond the degree
sequence have been discarded.

The defect analysis makes the remaining
combinatorial challenge substantially sharper. Whenever a graphic
Havel--Hakimi trajectory reaches a unit-band step, every subsequent positive
step is also a unit-band step, and the sum $T_n$ of the subsequent defect
increments is strictly less than $3/2$ by
Theorem~\ref{thm:terminal-band}. The sum $B_n$ of the preceding increments is
bounded above by the initial-degree quantity $E_n$ in
Lemma~\ref{lem:static-quantile}. The sharp $+2$ problem is therefore reduced
to proving $B_n\leq3/2$ for $n\geq15$. The weaker objective of proving any
uniform upper bound on $E_n$ would establish that
$\R(G_n)-\CW(G_n)$ is bounded independently of $n$. Either statement would
show that the residue itself, and
not only the Caro--Wei sum, carries an estimate equivalent to the Riemann
hypothesis.

\section*{CRediT authorship contribution statement}
Randy Davila: Conceptualization, Methodology, Formal analysis, Investigation,
Software, Validation, Visualization, Writing -- original draft, Writing --
review and editing.

\section*{Funding}
No funds, grants, or other support were received for this work.

\section*{Declaration of competing interest}
Randy Davila is affiliated with FirstPrinciples Inc., which develops the
TheoX research system described in the methodological section. The author
declares no other competing financial or personal interests that could have
influenced the work reported in this article.

\section*{Data and code availability}
The finite data reported in this article were generated by exact computation.
The Supplementary Material contains the source code needed to reproduce the
reported identity checks and selected trajectory calculations, the complete
trajectory tables for $4\leq n\leq10{,}000$, selected trajectories through
$10^6$, and a machine-readable audit report. Retained TheoX run records and
additional development artifacts are
available from the author on reasonable request.

\section*{Declaration of generative AI and AI-assisted technologies in the
writing process}
During the preparation of this work, the author used the TheoX research
system, incorporating OpenAI language-model services, for the research and
writing activities described in Section~7, including candidate generation,
computational checking, searches for gaps in draft proofs, and language
revision. The author
reviewed and edited all resulting content, independently checked every proof,
citation, and computation, and takes full responsibility for the published
article.

\begingroup
\small
\bibliographystyle{elsarticle-num}
\bibliography{davila_common_divisor_residue_ejc}
\endgroup

\end{document}